\documentclass[10pt,journal,final,a4paper,oneside,twocolumn]{IEEEtran}
\usepackage{epsfig}
\usepackage{amsmath,epsfig,amssymb,enumerate,tabularx}

\newtheorem{theorem}{Theorem}
\newtheorem{definition}{Definition}
\newtheorem{lemma}{Lemma}

\newcommand{\mathcalbold}[1]{\ensuremath{\boldsymbol{\mathcal{#1}}}}
\newcommand{\transconj}[1]{\ensuremath{{#1}^{\mathrm{\scriptscriptstyle{H}}}}}

\title{Quasi-convexity of the asymptotic channel MSE in regularized semi blind estimation}
\author{\IEEEauthorblockN{Abla Kammoun\IEEEauthorrefmark{1}(Contact Author), Karim Abed-Meraim\IEEEauthorrefmark{1}$^{,}$ and Sofi\`ene Affes\IEEEauthorrefmark{5}\\}
\IEEEauthorblockA{\IEEEauthorrefmark{1}{\small T\'{e}l\'{e}com-ParisTech 46, rue Barrault, 75634 Paris Cedex 13, France}\\}
\IEEEauthorblockA{\IEEEauthorrefmark{5}{\small INRS-EMT, 800, de la Gaucheti\`ere, Montreal, QC, H5A 1K6, Canada\\}
\small Email: abed, kammoun@tsi.enst.fr, affes@emt.inrs.ca}}
\begin{document}

\maketitle

\begin{abstract}
In this paper,  the quasi-convexity of a sum of quadratic fractions in the form $\sum_{i=1}^n \frac{1+c_i x^2}{\left(1+d_ix\right)^2}$ is demonstrated  where $c_i$ and $d_i$ are strictly positive scalars, when defined on the positive real axis $\mathbb{R}^{+}$. It will be shown that this quasi-convexity guarantees it has a unique local (and hence global) minimum.

Indeed, this problem arises when considering the optimization of the weighting coefficient in regularized semi-blind channel identification problem, and more generally, is of interest in other contexts where we combine two different estimation criteria. 

Note that V. Buchoux {\it et.al} have noticed by simulations that the considered function has no local minima except its unique global minimum but this is the first time this result, as well as the quasi-convexity of the function is  proved theoretically.

\IEEEkeywords{AAsymptotic analysis, Channel estimation, Exponential polynomial, Minimum MSE, Quasi-convexity,  Regularization, Semi-blind estimation }
\end{abstract}
\section{Introduction}
Many parameter estimation techniques use combined criteria to exploit different features or properties of the considered signals and hence improve the estimation performance. Examples of such combined techniques include the blind source separation (BSS) method in \cite{klajman}, and the blind equalization method in \cite{duhamel} where second and higher order statistics based criteria are combined to restore the source signals, and   the channel equalization and offset estimation technique in \cite{djebbar}, where again two criteria based on two different features of the transmit signals are jointly used to improve the receiver performance. 

In \cite{taoufik}, a similar approach is used for channel shortening in OFDM systems, and in \cite{geli,palicot}, semi-blind channel identification methods are considered where data-aided and blind techniques are combined together to shorten the training sequence while preserving a high channel estimation quality.

When combining two criteria, one uses a weighting parameter that needs to be optimized. 
In \cite{buchoux}, the weighting coefficient is optimized in such a way the asymptotic mean square error (MSE) of the channel estimate is minimum. 

The latter is shown to be a non-linear function and its optimization in \cite{buchoux} is done numerically using a line search algorithm. 

In this paper, we demonstrate that the previous asymptotic MSE function is quasi-convex which provides a guarantee that numerical optimization always leads to the desired optimal weighting parameter value. 

Moreover, one can observe that many asymptotic MSE functions have similar forms to that in \cite{buchoux} and thus, we believe the result given in this paper might be extended and adapted to other problems, where an optimal weighting coefficient is needed to combine two contrast functions. As an example, we can cite the case where a contrast function is linearly combined with an  MSE criterion. Referring to the work in \cite{ICASSP.cardoso.93}, we can easily show that in this case, the expression of the asymptotic MSE has the same form as the one described in \cite{buchoux}.

From the mathematics point of view, our work can be viewed as a contribution to the study of the roots of real exponential polynomials. It should be noted that this issue has been studied for general cases in \cite{Wielonsky.04}, where interesting results about the number of roots of real exponential polynomials with real frequencies have been presented. Unfortunately, these results yield a loose bound for the number of roots of the considered exponential polynomial and thus are of no interest for our particular case given by \eqref{eq:f}. We have provided in our work an original proof that takes into consideration the specifications of the considered exponential polynomial. 

The paper is organized as follows. Section II summarizes the results in \cite{buchoux} and shows that the considered MSE optimization problem can be cast into the optimization problem of a sum of quadratic fractions of the form:
$$\sum_{i=1}^n \frac{1+c_i x^2}{\left(1+d_ix\right)^2}, \hspace{0.1cm} c_i>0, d_i> 0.
$$
Section III and IV are completely devoted to the derivation and the proof of the quasi-convexity of the asymptotic MSE function. In particular, section III contains  some basic notions and results about quasi-convex functions. 
Conclusion and final remarks are given in section V.

{\bf Notation:} Operators $^{\tiny \mbox{H}}$, $^{-1}$ and $\mathrm{Tr}$ denote Hermitian, matrix inversion  and trace operators. Moreover, the real  and imaginary parts of a complex $z$ are denoted respectively by $\mathrm{Re}(z)$ and $\mathrm{Im}(z)$.

\section{Regularized semi-blind channel estimation}
In many signal processing applications, the major problem is to find out how to estimate some parameters at a low cost and with a good accuracy. The best estimate  we can have is obtained by taking into account all the information that we can get about the desired parameter. This approach involves in general high computational complexity, thus restricting its interest to only theoretical issues. 

Actually, in practice, suboptimal approaches retain only one kind of information on which they are based to derive a minimization problem with only a single criterion. 
An intermediate approach that is based on linearly combining competitive criteria has been recently proposed in many signal processing applications. For instance, in \cite{taoufik}, the channel shortening is being improved by linearly combining the null tones criterion with that of the guard interval. Also in the context of estimating sparse parameter vectors, $\mathcal{L}_p$ (where $p<1$) quasi-norms are often linearly combined to standard statistical criteria, thus allowing to take into account the sparsity of the desired solution \cite{jalil}. 

The optimal selection of the regularizing parameter that makes the best trade-off between two different criteria is important for the considered parameter estimation problem.  To the best of our knowledge, the minimization of the  mean square estimation error  with respect to the regularizing coefficient has been only analysed rigorously in \cite{buchoux}, in the context of semi-blind channel estimation. 

Since we will heavily rely on the asymptotic expression derived in \cite{buchoux}, it may be illuminating to provide a brief overview on the regularized semi-blind estimation technique. 

Regularized semi-blind estimation technique combines blind and training based criteria. They have been introduced first in the context of Single Input Multiple Output (SIMO) systems . In this case, if ${s}_k$ denotes the unit power transmitted signal, the vector received by the $N$ receiving antennas ${\bf y}_k$ is given by:
\begin{equation}
{\bf y}_k=\sum_{l=0}^L {\bf h}_l{s}_{k-l}+{\bf v}_k
\end{equation}
where ${\bf h}_l$ is the $N\times 1$ vector of the  $l$-th tap of the channel impulse response and ${\bf v}_k$ denotes the additive Gaussian noise. 
We assume that each frame is composed of training and data period, (see fig \ref{fig:multiplex}).
\begin{figure}[htbp]
\begin{center}
 \includegraphics[scale=0.5]{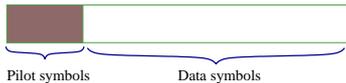}
\caption{Time-Multiplexed Training Scheme}
\label{fig:multiplex}
\end{center}
\end{figure}
 The training period corresponds to the transmission of $\ell$ known symbols which are often referred  to as pilots, whereas the data period corresponds to the transmission of $p$ data symbols.

The blind criterion is based on the statistical properties of the received signal in the data period and can be put on the form:
\begin{equation}
\min_{\|{\bf h}\|=1} \transconj{\bf h}{\bf Q}{\bf h}
\label{eq:blind}
\end{equation}
where ${\bf h}$ is the channel parameter and ${\bf Q}$ is a matrix that depends solely on the statistical properties of the received signal. 
On the other hand, the training based criterion can be expressed as;
\begin{equation}
\min_{{\bf h}}\|{\bf y}-{\bf Sh}\|^2
\label{eq:training}
\end{equation}
where ${\bf y}$ is the received signal and ${\bf S}$ is a matrix that depends on the pilot symbols. 
In contrast to blind estimation methods, training based techniques are more sensitive to noise and entail inefficient bandwidth utilization. However, blind methods are more complex, estimate the channel only up to a scalar ambiguity and are often non-robust to modelization errors (e.g. channel order overestimation errors) \cite{Zeng}). For these reasons, it might be interesting to combine linearly both criteria so as to resolve the drawbacks inherent to blind and training based techniques. Hence, the semi-blind  estimate is the one that minimizes:
\begin{equation}
\min_{{\bf h}}\|{\bf y}-{\bf Sh}\|^2+\lambda p\transconj{\bf h}{\bf Q}{\bf h}
\label{eq:semi_blind}
\end{equation}
where $\lambda>0$ is the regularizing coefficient and $p$ is the length of the information sequence.
Note that the semi-blind approach in \eqref{eq:semi_blind} outperforms the blind approach in \eqref{eq:blind} and the non-blind approach in \eqref{eq:training} only if the regularizing scalar $\lambda$ is chosen properly. In particular, this would be the case if $\lambda$ is selected in such a way the asymptotic estimation error variance is minimized. Our result is also useful to derive a relation between the optimal MSE$^*$ and the percentage of training symbols which can be adjusted to achieve a target MSE performance. 

It has been proved in \cite{buchoux} that the trace of the asymptotic\footnote{Asymptotic refers to the case where $p\to\infty$, $\ell\to\infty$ and the ratio $\frac{p}{\ell}\to\gamma$.} estimation mean-square error (MSE) is proportional to:
$$
\mathrm{MSE} \hspace{0.1cm} \propto \hspace{0.1cm} \mathrm{Tr}\left\{\left({\bf I}+\lambda \gamma{\bf Q}\right)^{-1}\left({\bf I}+\lambda^2\gamma\mathcalbold{M}({\bf h})\right)\left({\bf I}+\lambda \gamma{\bf Q}\right)^{-1}\right\}
$$
where $\gamma=\frac{p}{\ell}$,  and $\mathcalbold{M}({\bf h})$ is a Hermitian matrix that has the same row and column space as ${\bf Q}$ (meaning that if ${\bf Q}={\bf U}{\bf D}\transconj{\bf U}$ is the eigenvalue decomposition of ${\bf Q}$, $\mathcalbold{M}({\bf h})$ writes as $\mathcalbold{M}({\bf h})={\bf U}{\bf A}{\bf }\transconj{\bf U}$, where ${\bf A}$ is a given Hermitian matrix.). 

Using the eigenvalue decomposition of ${\bf Q}$, it can be easily verified that the $\mathrm{MSE}$ is proportional to:
\begin{equation}
\mathrm{MSE}\hspace{0.1cm} \propto \hspace{0.1cm} \sum_{i}\frac{1+\lambda^2\gamma{ a}_{ii}}{\left(1+\lambda\gamma{ d}_{ii}\right)^2}
\end{equation}
where ${a}_{ii}>0 $ (resp. ${d}_{ii}>0$) denote the diagonal elements of ${\bf A}$ (resp. the non zero diagonal elements of ${\bf D}$).

Note that in \cite{buchoux} and \cite{gretsi}, it was noticed by simulations that the $\mathrm{MSE}$ has a unique local (global) minimum with respect to $\lambda$, but to the best to our knowledge, until now, this result has not been proved in any previous work.

\section{Quasi-convexity of the Mean square error}
For the reader convenience, we recall hereafter the definition and also some results about quasi-convex functions. (we refer the reader to \cite{boyd} for further information).
\begin{definition}

 A  real valued function $f$ is said to be quasi-convex if its domain of definition and all its sublevel sets:
$$
S_{\alpha}=\left\{x\in \mathrm{\bf dom}f | f(x)\leq \alpha\right\}
$$
for $\alpha\in\mathbb{R}$, are convex, where $\mathrm{\bf dom}f$ denotes the set over which the function $f$ is defined.
\end{definition}
{\bf Examples of quasi-convex functions}
To illustrate this concept, we provide in fig \ref{fig:example} some examples of quasi-convex functions. As we can see, we  note that a concave and also a non convex function can be also quasi-convex. 
\begin{figure}
 \begin{center}
  \includegraphics[scale=0.4]{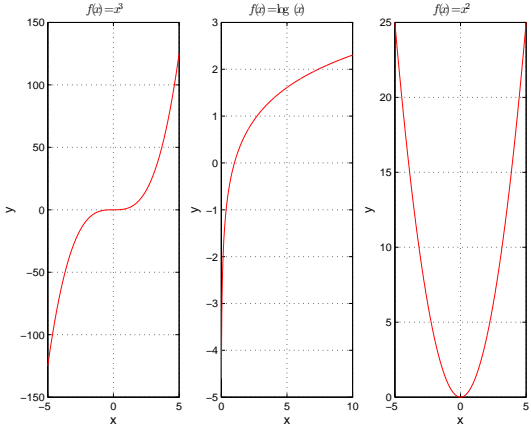}
\caption{Examples of quasi-convex functions}
\label{fig:example}
 \end{center}
\end{figure}

Like convex functions, quasi-convex functions satisfy a modified Jensen inequality which is given by:

\begin{theorem}
A function $f$ is quasi-convex if and only if $\mathrm{\bf dom}f$ is convex and for all $x,$ $y\in  \mathrm{\bf dom}(f)$ and $0\leq \theta \leq 1$
$$
f(\theta x+ (1-\theta) y)\leq \mathrm{max}\left\{f(x),f(y)\right\}.
$$
\end{theorem}
Clearly, the quasi-convexity generalizes the notion of convexity in the sense that the class of quasi-convex functions is larger than and includes the class of convex functions. Also, in most cases, quasi-convex functions inherit  the nice properties of convex functions including the absence of local minimum as stated in the following theorem.
\begin{theorem}
 Let $f$ be a quasi-convex function. Then every local minimum is a global minimum or $f$ is constant in a neighborhood of this local minimum. 
\end{theorem}
Consequently, if a quasi-convex function $f$ is non constant over any given interval (which is the case for the sum of quadratic functions we consider), then each local minimum is also a global minimum. Moreover, this global minimum (whenever it exists) is unique for real valued functions.
To prove the non existence of local minima besides the global one, we use often the following second-order condition:
\begin{theorem}
 Let $f$ be a real function which is twice derivable. If $f$ satisfies:
$$
\forall c \hspace{0.3cm}\textnormal{such that}\hspace{0.3cm} f'(c)=0, \hspace{0.2cm} f''(c)> 0,
$$
then, $f$ is quasi-convex, and each local minimum is a global minimum. 
\label{th:second}
\end{theorem}
Next we state our main result regarding the unimodality of the asymptotic MSE then we prove it in the section after. 

\begin{theorem}
Let $c_i$, $d_i$ be two sequences of $n\in\mathbb{N}^*$ strictly positive reals. Then the derivative of  

\begin{equation}
F_n(x)=\sum_{i=1}^n \frac{1+c_i x^2}{\left(1+d_ix\right)^2}
\label{eq:Fn}
\end{equation}
has a unique positive zero $x_0$ with $F^{(2)}_n(x_0)>0$. Consequently, $F_n(x)$ is a quasi-convex function when its domain of definition is restricted to $\mathbb{R}^+$ and hence has a unique local (global) minimum on the positive real axis. In the sequel, we will omit the index $n$ for notational simplicity so that $F_n$ will be referred to as $F$.
\label{th:unique}
\end{theorem}

To prove this theorem, we proceed in the following steps.
\begin{itemize}
\item First, we show that the number of positive real values of $F^{(k+1)}$ is larger or equal than that of $F^{(k)}$, where $F^{(k)}$ denotes the $k$-th derivative of $F$. 
\item We introduce the function $G_{k}$ which has the same number of zeros as $F^{(k)}$ and prove that it converges uniformly to $G_{\infty}$, over a compact set that contains all the zeros of $F^{(k)}$. 
\item Then we prove that  $G_{\infty}$ has a unique positive zero in that compact set.
\item By applying Hurwitz theorem \cite{remmert}, we conclude that for large values of $k$, $G_{k}$ is zero only once and that will be also the case of $F^{(k)}$.
\item Finally, we prove that the second derivative of $F$ is strictly positive when evaluated at the zero argument of $F$.
Fig \ref{fig:F} illustrates the shape of function $F$ and its first and second order derivatives, for $n=3$.
\begin{figure}[htbp]
\begin{center}
 \includegraphics[scale=0.4]{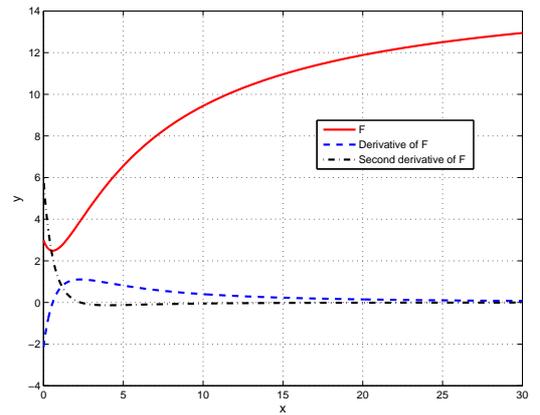}
\end{center}
\caption{Function F and its first and second order derivatives}
\label{fig:F}
\end{figure}
\end{itemize}
Next section provides the  details of all these steps and their proofs.


\section{Analysis and Properties of $F$}
\subsection{Closed-form expressions for the derivatives of $F$}
In this subsection, we provide a closed form expression for the $k$-th derivative of  function $F$. We also show that the number of zeros of the $k$-th derivative is increasing  with $k$.
\begin{lemma}
\label{lemma:Fn}
The $k$-th derivative of $F(x)$ ($k>0$) can be put on the following expression:
\begin{equation}
F^{(k)}(x)=(-1)^{k+1}\sum_{i=1}^n\frac{b_{i,k}x-a_{i,k}}{(1+d_ix)^{k+2}}
\label{eq:F}
\end{equation}
where $a_{i,k}$ and $b_{i,k}$ are  sequences of positive reals given by:
$$
\left\{
\begin{array}{lll}
b_{i,k}&=&2k!c_id_i^{k-1}\\
a_{i,k}&=&2k!c_id_i^{k-1}\left(\frac{k+1}{2}\frac{d_i}{c_i}+\frac{k-1}{2d_i}\right)
\end{array}
\right.
$$
\end{lemma}
\begin{proof}
 See Appendix \ref{appendix:Fn}.
\end{proof}
Given the previous expressions of $F^{(k)}$, we are able to prove our first step result concerning the increasing number of zeros of $F^{(k)}$. We have the following lemma:
\begin{lemma}
Let $Z_k$ denote the number of zeros of the $k$-th derivative of $F$ given by  (\ref{eq:F}). Then $Z_{k+1}\geq Z_k$.
\label{lemma:Zk}
\end{lemma}
\begin{proof}
Let $x_1,\cdots,x_{\scriptscriptstyle{Z_k}}$ denote the zeros of the $k$-th derivative $F^{(k)}$ in $\left[0,\infty\right[$. Therefore, using Rolle's Theorem \cite{Clapham}, $F^{(k+1)}$ has at least $Z_{k}-1$ zeros $y_1,\cdots,y_{\scriptscriptstyle{Z_k-1}}$ where $x_i\leq y_{i}\leq x_{i+1}, i\in \left\{1,\cdots,Z_{k}-1\right\}$.
Since $\lim_{x\rightarrow+\infty}F^{(k)}(x)=0$, there exist at least one zero of $F^{(k+1)}$ in $\left[x_{\scriptscriptstyle{Z_k}},\infty\right[$. Consequently, the number of zeros of $F^{(k+1)}$ is at least equal to $Z^{(k)}$, i.e $Z_{k+1}\geq Z_k$.
\end{proof}
\subsection{Uniform equivalence of $G_{k}$}
In this subsection, we introduce an alternative function $G_k$ that has the same number of positive valued zeros as $F^{(k)}$ and we provide its asymptotic equivalent expression. For that, let us start by providing a useful approximation of coefficient $a_{i,k}$ that will be used later to build the function $G_k$.

The Stirling formulae \cite{Stirling} provides us an equivalent \footnote{Equivalence here means that $\frac{k!}{\sqrt{2\pi k}\left(\frac{k}{e}\right)^k}\xrightarrow[k\to\infty]{}1$} for $k!$:
$$
k!\sim \sqrt{2\pi k}\left(\frac{k}{e}\right)^k
$$
we can easily show that:
\begin{align}
a_{i,k}&\sim \sqrt{2\pi} k^{k+\frac{3}{2}}e^{-k}c_id_i^{k-1}\left(\frac{d_i}{c_i}+\frac{1}{d_i}\right)\nonumber\\
&\sim\sqrt{2\pi} k^{k+\frac{3}{2}}e^{-k}c_id_i^{k+2}\left(\frac{1}{c_id_i^2}+\frac{1}{d_i^4}\right)\label{eq:equivalence}.\\
\nonumber
\end{align}

We recall that the overall quasi-convexity proof is based on  studying  the zeros of the function $F^{(k)}$ as $k$ goes to infinity. Actually, one can show\footnote{Outside this interval, all the terms in the sum given by (\ref{eq:F}) have the same sign and hence $F^{(k)}$ cannot be zero.} that these zeros belong to  the interval $\left[V_{\mathrm{min}}^k,V_{\mathrm{max}}^k\right]$, where $V_{\mathrm{min}}^k=\displaystyle\min_{i \in\{1,\cdots n\}}\frac{a_{i,k}}{b_{i,k}}$ and $V_{\mathrm{max}}^k=\displaystyle\max_{i \in\{1,\cdots n\}}\frac{a_{i,k}}{b_{i,k}}$. A lower bound for $V_{\mathrm{min}}^k$ and an upper bound for  $V_{\mathrm{max}}^k$ can be easily computed and are given by:
\begin{align}
V_{\mathrm{min}}^k&\geq k\tau_{\mathrm{min}}\\
V_{\mathrm{max}}^k&\leq k\tau_{\mathrm{max}}  
\end{align}
where $\displaystyle \tau_{\mathrm{min}}=\min_{i\in \{1,\cdots n\}}\frac{d_i}{2c_i}$ and $\tau_{\mathrm{max}}=\displaystyle\max_{i \in \{1,\cdots n\}}\frac{d_i}{c_i}+\frac{1}{2d_i}$.
The difficulty that we face is that the zeros of $F^{(k)}$ are of order $k$, thus making the analysis of the asymptotic behavior of  function $F^{(k)}$ somehow delicate. 
To deal with this difficulty,  another function, which we denote by $G_{k}$, and which brings back those zeros to a given fixed interval is introduced.  This function will be studied over the interval of interest $\left[\tau_\mathrm{min},\tau_{\mathrm{max}}\right]$.
 
 Function $G_k$ is defined as:
\begin{equation}
G_{k}(x)=(-1)^{k+1}\sqrt{\frac{k}{2\pi}}e^{k}x^{k+2}F^{(k)}(kx)
\label{eq:function_gk}
\end{equation}

One can easily note that over $\left[\tau_\mathrm{min},\tau_{\mathrm{max}}\right]$, $G_{k}(x)$ has the same number of zeros as $F^{(k)}$. 

Clearly, the scaling of the variable $x$ by factor $k$ is introduced to bring back the roots of $F^{(k)}$ from the interval $\left[k\tau_{\mathrm{min}},k\tau_{\mathrm{max}}\right]$ to the finite length interval $\left[\tau_{\mathrm{min}},\tau_{\mathrm{max}}\right]$. The multiplicative function in \eqref{eq:function_gk} (i.e $\sqrt{\frac{k}{2\pi}}e^{k}x^{k+2}$) is introduced to normalize the coefficient $a_{i,k}$ and $b_{i,k}$ and to approximate the denominator terms in \eqref{eq:F} by exponential functions. 

Substituting $F^{(k)}$ by its expression in (\ref{eq:F}), $G_{k}$ writes as:

\begin{align*}
G_{k}(x)&=\sqrt{\frac{k}{2\pi}}e^kx^{k+2}\sum_{i=1}^n\frac{kb_{i,k}x-a_{i,k}}{\left(1+kd_ix\right)^{k+2}}\\
&=\sqrt{\frac{k}{2\pi}}e^k\sum_{i=1}^n\frac{a_{i,k}\left(\frac{kb_{i,k}}{a_{i,k}}x-1\right)}{k^{k+2}d_i^{k+2}\left(\frac{1}{kd_ix}+1\right)^{k+2}}\\
&\triangleq\sum_{i=1}^{n}\frac{g_{i,k}(x)}{h_{i,k}(x)}
\end{align*}

where $g_{i,k}(x)\triangleq \frac{\sqrt{k}e^k}{\sqrt{2\pi}d_i^{k+2}k^{k+2}}a_{i,k}\left(\frac{kb_{i,k}}{a_{i,k}}x-1\right)$ and $h_{i,k}(x)=(\frac{1}{kd_ix}+1)^{k+2}$.

In the following, we extend the domain of the function $G_{k}$ to the rectangle $R$ of $\mathbb{C}$ given by:
$$\mathcal{R}_\epsilon=\left\{z=x+iy, x\in \left[\tau_\mathrm{min},\tau_{\mathrm{max}}\right], -\epsilon\leq y\leq \epsilon \right\}$$ where $\epsilon$ is a constant real that will be specified later. Over this domain, the asymptotic equivalent of $G_k$ is given by the following theorem:
From the previously stated lemma, one can prove easily the following result:

\begin{theorem}
In the rectangle $\mathcal{R}_{\epsilon}$,  $G_{k}$ converges uniformly to $G_{\infty}$ given by:
$$
G_{\infty}(z)=\sum_{i=1}^nc_i\left(\frac{1}{c_id_i^2}+\frac{1}{d_i^4}\right)\left(V_{\infty,i}z-1\right)e^{-\frac{1}{d_iz}}
$$
where $V_{\infty,i}=\frac{2}{\frac{d_i}{c_i}+\frac{1}{d_i}}$.
\label{th:convergenceG}
\end{theorem}
\begin{proof}
See Appendix \ref{appendix:convergenceG}.
\end{proof}
\subsection{Zeros of the uniform limit of $G_k$}
In this section, we prove that $G_{\infty}$ has a unique positive real zero. 
This is a byproduct of the following theorem:
\begin{theorem}
Let $a_i$ $b_i$ and $\alpha_i$ three sequences of $n$ strictly positive real scalars. Let $f$ be the function given by:
\begin{equation}
f(x)=\sum_{i=1}^n(a_ix-b_i)e^{-\alpha_ix}
\label{eq:f}
\end{equation}
Then $f$ admits a unique real positive zero.
\label{th:dif}
\end{theorem}
\begin{proof}
See Appendix \ref{appendix:dif}.
\end{proof}
By defining $f(z)=zG_{\infty}(\frac{1}{z})$ and applying Theorem \ref{th:dif}, we conclude that $G_{\infty}(z)$ has a unique real positive zero. 
\subsection{Application of Hurwitz theorem}
To prove that from a certain range of $k$ $G_k$ is zero only once at the real positive axis, we will rely on the following known result in complex analysis, \cite{remmert}:
\begin{theorem}
Let  $f_k(z)$ be a sequence of analytic functions in a compact ${\bf C}$. Assume that $f_k$ converges uniformly to $f$ in ${\bf C}$. Assume also that $f$ has no zeros on the frontier $\partial{\bf C}$ of ${\bf C}$. Then, there exists $k_0\in\mathbb{N}$ such that $\forall k\geq k_0$, $f$ and $f_k$ have the same number of zeros in ${\bf C}$.
\end{theorem}
Applying this theorem, we can deduce that, $G_k$ will have a unique zero value in $\mathcal{R}_{\epsilon}$ as $G_{\infty}$, where $\epsilon$ is chosen so that $G_{\infty}$ has no zeros on the frontier of $\mathcal{R}_{\epsilon}$ and has no complex zeros besides its real positive zero. Since the number of zeros of $G_k$ is increasing with respect to $k$, we conclude that all $G_k$ and hence all $F^{(k)}$ have only a unique positive zero.

Let $x_z$ be the unique positive zero argument of $F^{(1)}$. Since $F^{(1)}$ is negative in a neighborhood of zero, and  $F^{(1)}$ has no zeros for $x\leq x_z$, $F^{(1)}$ is negative in the interval $\left[0, x_z\right]$. Therefore the function $F$ is decreasing in $\left[0, x_z\right]$. 

Since $F^{(1)}$ is positive for large value of $x$, $F^{(1)}$ must change its sign at $x_z$, and hence it is positive in the interval $\left[x_z,\infty\right[$. Consequently, $F$ is increasing in $\left[x_z,\infty\right[$. 

To sum up, we have established that in $\left[0, x_z\right]$, $F$ is decreasing and in $\left[x_z,\infty\right[$ $F$ is increasing. This guarantees that $x_z$ is a minimum for $F$ and hence $F^{(2)}(x_z)\geq0$. In fact, $F^{(2)}(x_z)$ is strictly positive, since $F^{(1)}(x_z)=0$ and $\displaystyle\lim_{x\to\infty}F^{(1)}(x)=0$ means that there exists $y_z\in\left]x_z,\infty\right[$ such that $F^{(2)}(y_z)=0$ and hence $F^{(2)}(x_z)\neq0$ (because $F^{(2)}$ has a unique zero).
%
%
 \section{Conclusion}
 In this paper, we have provided a rigorous proof for the quasi-convexity of the asymptotic MSE  of the regularized semi-blind channel estimate.

More generally, we have proved that any function given by a finite sum of quadratic fractions $\frac{1+cx^2}{\left(1+dx\right)^2}, c, d >0$ is a unimodal function over $\mathbb{R}^{+}$.

For our considered channel estimation problem, the previous result guarantees the absence of non-desired local minima of the MSE function when optimized with respect to the weighting coefficient.
\appendices
\section{Proof of lemma \ref{lemma:Fn}}
\label{appendix:Fn}
\begin{proof}
Lemma \ref{lemma:Fn} can be proved easily by induction on $k$. For $k=1$, we have:

\begin{align*}
F^{(1)}(x)&=\sum_{i=1}^n\frac{2c_ix(1+d_ix)^2-2d_i(1+c_ix^2)(1+d_ix)}{(1+d_ix)^4}\\
&=\sum_{i=1}^n\frac{(1+d_ix)(2c_ix(1+d_ix)-2d_i(1+c_ix^2))}{(1+d_ix)^4}\\
&=\sum_{i=1}^n2\frac{c_ix-d_i}{(1+d_ix)^3}
\end{align*}
Let  $k \in \mathbb{N}^*$. Assume that the result is true until order $k$. Hence, $F_n^{(k)}$ can be written as:
$$
F^{(k)}(x)=(-1)^{k+1}\sum_{i=1}^n\frac{b_{i,k}x-a_{i,k}}{(1+d_ix)^{k+2}}
$$
Therefore,
\begin{align*}
F^{(k+1)}(x)&=(-1)^{k+1}\left(\sum_{i=1}^n\frac{b_{i,k}(1+d_ix)^{k+2}}{(1+d_ix)^{2k+4}}\right.\\
&\left.-\frac{(k+2)d_{i}(1+d_ix)^{k+1}(b_{i,k}x-a_{i,k})}{(1+d_ix)^{2k+4}}\right)\\
\!\!&\!\!\!\!=(-1)^{k+1}\sum_{i=1}^n\frac{b_{i,k}+(k+2)d_ia_{i,k}-(k+1)d_ib_{i,k}x}{(1+d_ix)^{k+3}}\\
&=(-1)^{k+2}\sum_{i=1}^n\frac{(k+1)b_{i,k}d_ix-(b_{i,k}+(k+2)d_ia_{i,k})}{(1+d_ix)^{k+3}}\\
&=(-1)^{k+2}\sum_{i=1}^n\frac{b_{i,k+1}x-a_{i,k+1}}{(1+d_ix)^{k+3}}
\end{align*}
where $b_{i,k+1}=(k+1)b_{i,k}d_i$ and $a_{i,k+1}=(b_{i,k}+(k+2)d_ia_{i,k})$.
Since $b_{i,k}=2k!d_i^{k-1}c_i$, we get $b_{i,k+1}=2(k+1)!d_i^{k}c_i$.

Also,
\begin{align*}
a_{i,k+1}&=b_{i,k}+(k+2)d_ia_{i,k}\\
&=2k!c_id_i^{k-1}+2k!(k+2)d_i^kc_i\left(\frac{k+1}{2}\frac{d_i}{c_i}+\frac{k-1}{2d_i}\right)\\
&=2(k+1)!c_id_i^k\left(\frac{1}{(k+1)d_i}+\frac{k+2}{k+1}\left(\frac{k+1}{2}\frac{d_i}{c_i}+\frac{k-1}{2d_i}\right)\right)\\
&=2(k+1)!c_id_i^k\left(\frac{k}{2d_i}+\frac{k+2}{2}\frac{d_i}{c_i}\right).
\end{align*}
\end{proof}
\section{Proof of theorem \ref{th:convergenceG}}
In this lemma, we propose to find the uniform-limit function for the function $G_k(z)=\sum_{i=1}^n\frac{h_{i,k}(z)}{g_{i,k}(z)}$ in the rectangle $\mathcal{R}_{\epsilon}$. For that, we will first begin by finding the uniform limit functions of  $h_{i,k}$ and $g_{i,k}$. 
\label{appendix:convergenceG}
\begin{lemma}
In the rectangle  $\mathcal{R}_{\epsilon}$, the sequence of functions $(h_{i,k})_k$ converges uniformly to $h_{i,\infty}$ given by:
$$
h_{i,\infty}(z)=e^{\frac{1}{d_iz}}
$$

Also, the sequence of functions $(g_{i,k})_k$ converges uniformly to $g_{i,\infty}$ given by:
$$
g_{i,\infty}=c_i\left(\frac{1}{c_id_i^2}+\frac{1}{d_i^4}\right)\left(V_{\infty,i}z-1\right)
$$
where $V_{\infty,i}=\lim_{k\to\infty}\frac{kb_{i,k}}{a_{i,k}}=\frac{2}{\frac{d_i}{c_i}+\frac{1}{d_i}}$.
\end{lemma}
\begin{proof}

The uniform convergence of $h_{i,k}$ to $h_{i,\infty}$ is a by-product of the following known result:
\begin{lemma}
Over a compact set the sequence function $(1+\frac{z}{n})^n$ converges uniformly to $e^z$.
\end{lemma}
The uniform convergence of $g_{i,k}$ to $g_{i,\infty}$ is obtained by using the asymptotic equivalent of $a_{i,k}$ given in (\ref{eq:equivalence}).
\end{proof}
The uniform convergence of $h_{i,k}$ to $h_{i,\infty}$ and of $g_{i,k}$ to $g_{i,\infty}$ does not ensure the uniform convergence of $\frac{g_{i,k}}{h_{i,k}}$ to $\frac{g_{i,\infty}}{h_{i,\infty}}$. 

Other extra conditions are needed as it will be noticed in the following lemma:

\begin{lemma}
\label{lemma:convuni}
Let $f_k$ and $g_k$ denote sequences of continuous functions over a compact $C$. Assume that $g_k$ is bounded over $C$ away  from zero uniformly  in $k$ and in $z$, i.e there exists a constant $M$ such that:
$$\forall k\in\mathbb{N}, \forall z \in C \hspace{0.1cm} |g_k(z)| > M.$$

Assume also that $f_k$ and $g_k$ converge uniformly to $f_{\infty}$ and $g_{\infty}$. Then, $\frac{f_k}{g_k}$ converges uniformly to $\frac{f_{\infty}}{g_{\infty}}$ over the compact $C$.
\begin{proof}
Since $f_k$ and $g_k$ are continuous, their uniform limits $f_{\infty}$ and $g_{\infty}$ are also continuous. Therefore, there exists constant reals $M_f$, $M_g$  such that: 

$$\forall z \in C,  \left|f_{\infty}(z)\right| \leq M_f \ \ \ \mathrm{and} \ \ \ \left|g_{\infty}(z)\right| \leq M_g. $$ 
Since for all $k\in\mathbb{N}$, $|g_k(z)| > M$, we have $|g_{\infty}(z)| > M$
To prove the uniform convergence of $\frac{f_k}{g_k}$ towards $\frac{f_{\infty}}{g_{\infty}}$, it is sufficient to prove that $\displaystyle\sup_{z\in C}\left|\frac{f_k}{g_k}-\frac{f}{g}\right|$ converges to zero as $k$ tends to infinity. 
We have:
\begin{align*}
\sup_{z\in C}\left|\frac{f_k}{g_k}-\frac{f}{g}\right|&=\sup_{z\in C}\left|\frac{f_kg-fg_k}{g_kg}\right|\\
&\leq \sup_{z\in C}\left|\frac{f_kg-fg_k}{M^2}\right|\\
&\leq \frac{1}{M^2}\left(\sup_{z\in C}\left|f_kg-fg\right|+\sup_{z\in C}\left|fg-fg_k\right|\right)\\
&\leq \frac{1}{M^2}\left(M_g\sup_{z\in C}\left|f_k-f\right|+M_f\sup_{z\in C}\left|g_k-g\right|\right) \\
&\xrightarrow[k\to\infty]{}{0}\\
\end{align*}
which proves that $\frac{f_k}{g_k}$ converges uniformly to $\frac{f_\infty}{g_{\infty}}$.
\end{proof}
\end{lemma}
Since $|h_{i,k}(z)|> 1$ over $\mathcal{R}_{\epsilon}$, $h_{i,k}$  satisfies the condition of lemma \ref{lemma:convuni}. Applying this lemma on the functions $g_{i,k}$ and $h_{i,k}$, we prove that $\frac{g_{i,k}}{h_{i,k}}$ converges uniformly to $\frac{g_{i,\infty}}{h_{i,\infty}}$. Consequently,
$G_k(z)=\displaystyle\sum_{i=1}^n\frac{g_{i,k}(z)}{h_{i,k}(z)}$ converges uniformly over $\mathcal{R}_{\epsilon}$ to $G_{\infty}(z)=\displaystyle\sum_{i=1}^n \frac{g_{i,\infty}(z)}{h_{i,\infty}(z)}$.

\section{Proof of Theorem \ref{th:dif}}
\label{appendix:dif}
The proof is performed by induction on $n$. For $n=1$, the result is straightforward. Let $n\in \mathbb{N}^*$ be a given integer, and assume that the result holds true for all $k\leq n$, and all functions $f$ of the form given by (\ref{eq:f}).
Assume that there exists $a_i$, $b_i$ and $\alpha_i$ three sequences of $n+1$ strictly positive real scalars such that the function 
$$
f(x)=\sum_{i=1}^{n+1}(a_ix-b_i)e^{-\alpha_ix}
$$
admits more than one positive zero. Let  $x_1$ be the first smallest zero of $f$ on $\mathbb{R}^{+}$, 

Without loss of generality, we can assume that all the $\alpha_i$ are two by two different and that $\alpha_{n+1}=\displaystyle\min_{1\leq i\leq n+1}\alpha_i$.
 Since $f$ is strictly negative in zero and is positive for large values of $x$,  $f$ should change its sign at at least one zero.  
In the following we will consider only the case when $f$ changes its sign at $x_1$. The other case coule be treated in the same  way. Let $x_2$ be the second smallest zero of $f$ on $\mathbb{R}^+$. Under this condition, we distinuish the following cases:
\begin{itemize}
\item $f$ changes its sign at $x_1$ and at $x_2$.
\item $f$ changes its sign only at $x_1$.
\end{itemize}
For the both cases, we can prove that the second derivative of
\begin{align*}
g_m(x)&=e^{\left(\alpha_{n+1}-\frac{1}{m}\right)x}f(x)\\
&=\sum_{i=1}^{n}(a_ix-b_i)e^{-\left(\alpha_i-\alpha_{n+1}+\frac{1}{m}\right)x}+\left(a_{n+1}x-b_{n+1}\right)e^{-\frac{x}{m}}, \\
&\hspace{0.2cm}\textnormal{for} \hspace{0.05cm}m\in\mathbb{N}^*
\end{align*}
has three zeros. More particularly, we have the following:

\underline{Case 1: $f$ changes its sign at $x_1$ and at $x_2$}

Since $f(0)<0$, $f(x)<0$ for $x \leq x_1$. Therefore, for $x\geq x_2$ and in the vicinity of $x_2$, $f(x) <0$ for $x \geq x_2$. Since $f(x)>0$ for $x$ large enough, $f$ should have a third zero $x_3>x_2$. 

For all integers $m$, we note that $f$ and $g_m$ have the same number of zeros. Using Rolles theorem, it can be proved that the derivative of $g_m$ which we denote $g^{(1)}_m$ and which is given by:
\begin{align}
g^{(1)}_m(x)&=\sum_{i=1}^n\left[-(\alpha_i-\alpha_{n+1}+\frac{1}{m})a_ix+b_i(\alpha_i-\alpha_{n+1}+\frac{1}{m})\right.\\
&\left.+a_i\right]\times e^{-(\alpha_i-\alpha_{n+1}+\frac{1}{m})x}+a_{n+1}e^{-\frac{1}{m}x}-\frac{1}{m}\left(a_{n+1}x\right.\\
&\left.-b_{n+1}\right)e^{-\frac{x}{m}} \label{eq:g1m}
\end{align}
has at least three  zeros, since $g^{(1)}_m(x)$ tends to zero as $x$ tends to infinity.

Also  again by using the Rolle's theorem, we conclude that the second derivative of $g_m$ denoted by $g^{(2)}_m(x)$ has at least two zeros.

\underline{Case 2: $f$ changes its sign at only one zero}
In this case, we can also prove that the first derivative of $g_m$ has three zeros. Actually, at $x_2$, the first derivative of $g_m$ must be also zero, since $x_2$ is a local minimum for $f$ and hence for $g_m$. As $g_m$ tends to zero when $x$ tends to infinity, $g^{(1)}_m$ has two zeros between $]x_1,x_2[$ and $]x_2,\infty[$. Consequently, in total, $g^{(1)}_m$ has at least three zeros, and therefore, the second derivative of $g_m$ denoted  has at least two zeros.

Taking the derivative of (\ref{eq:g1m}), $g^{(2)}_m(x)$ writes as:
\begin{align*}
g^{(2)}_m(x)&=\sum_{i=1}^n(\alpha_i-\alpha_{n+1}+\frac{1}{m})\left[a_i(\alpha_i-\alpha_{n+1}+\frac{1}{m})x\right.\\
&\left.-2a_i-b_i(\alpha_i-\alpha_{n+1}+\frac{1}{m})\right]e^{-(\alpha_i-\alpha_{n+1}+\frac{1}{m})x}\\
&-\frac{2}{m}a_{n+1}e^{-\frac{1}{m}x}+\frac{1}{m^2}(a_{n+1}x-b_{n+1})e^{-\frac{x}{m}}
\end{align*}
Extending the definition domain of $g^{(2)}_m$ to $\mathbb{C}^+=\left\{z=x+iy, x >  0\right\}$, we note that for every compact in $\mathbb{C}^+$, $g^{(2)}_m$ converges uniformly to $g_{\infty}$ given by:
\begin{align*}
g_{\infty}(z)&=\sum_{i=1}^n(\alpha_i-\alpha_{n+1})\left(a_i(\alpha_i-\alpha_{n+1})z-2a_i\right.\\
&\left.-b_i(\alpha_i-\alpha_{n+1})\right)e^{-(\alpha_i-\alpha_{n+1})z}
\end{align*}
Let $\mathcal{C}$ be the contour corresponding to the rectangle 
\begin{align*}
R_{\epsilon}&=\left\{x+iy, x\in \left[\displaystyle{\inf_{m,i}}\frac{2a_i+b_i(\alpha_i-\alpha_{n+1})+\frac{1}{m}}{a_i(\alpha_i-\alpha_{n+1}+\frac{1}{m})},\right.\right.\\
&\left.\left.\displaystyle{\sup_{m,i}}\frac{2a_i+b_i(\alpha_i-\alpha_{n+1})+\frac{1}{m}}{a_i(\alpha_i-\alpha_{n+1}+\frac{1}{m})}\right], y\in\left[-\epsilon,\epsilon\right]\right\},
\end{align*}
$\epsilon$ is chosen such that $|g_{\infty}|$ is bounded above zero in $\mathcal{C}$ and has no complex valued zeros. 
 Then referring to Hurwitz theorem, $g_{\infty}$ and $g^{(2)}_m$ will have the same number of zeros in $R_\epsilon$ for large enough values of $m$, which is in contradiction with the induction assumption.
\bibliographystyle{IEEE}
\bibliography{./biblio}

\begin{biographynophoto}{Abla Kammoun}
 was born in Sfax, Tunisia. She received the Engineering degree in Signal and Systems from the Tunisia Polytechnic School, the master degree and the Phd degree in digital communications from Telecom Paris Tech (then Ecole Nationale Sup\'erieure des T\'el\'ecommunications, ENST). 

Since June 2010, she is a post-doctoral researcher in TSI department in Telecom Paris Tech. 

Her research interests include performance analysis, random matrix theory and semi-blind channel estimation.
\end{biographynophoto}
\begin{biographynophoto}{Karim Abed-Meraim}
Karim ABED-MERAIM was born  in 1967. He received the State Engineering Degree  from
Ecole Polytechnique, Paris, FRANCE, in 1990, the State Engineering Degree from  
Ecole Nationale Supérieure des Télécommunications (ENST), Paris, FRANCE, in 1992,
the M.Sc. degree from Paris XI University, Orsay, FRANCE, in 1992  and the Ph.D
degree from the Ecole Nationale Supérieure des Télécommunications (ENST), Paris,
FRANCE, in 1995 (in the field of Signal Processing and communications).

From 1995 to 1998, he has been  a research staff at the Electrical Engeneering
Department of the University of Melbourne where he worked on several research
project related to "Blind System Identification for Wireless Communications", "Blind
Source Separation", and "Array Processing for Communications", respectively.

He currently is Associate Professor  at the Signal and Image Processing Department
of Telecom ParisTech. His research interests are in  signal processing for
communications and include system identification, multiuser detection, space-time
coding, adaptive filtering and tracking, array processing and performance analysis.
He has about 300 scientific publications including 60 journal papers, 6 book
chapters and 4 patents. 

Dr Karim ABED-MERAIM was is charge of the  "Equalization" research operation within
the GT9/GDR-PRC ISIS (a research operation within the National Scientific Research
Centre of France (CNRS)) and was  leader of a research team within Telecom ParisTech
working on "Signal Processing for Communications".

Within his research activity, he has organized and chaired several conference
sessions (ICOTA'98, DSP'98 workshop, ISPCS'98, ISSPA'99, SSP2001, ISCCSP'04,
ISSPA'05, ...). He is  referee for several journals including IEEE-Tr-SP,
IEEE-Tr-IT, Sig. Proc. Journal, IEEE-SP-Letters, JSAC, and IEEE-Tr-Comm. He was the
technical Chairman of ISSPA'2001 and ISSPA'2003 (International Symposium on Signal
Processing and its Applications) ans  the conference Chair of WOSPA 2008.

Dr. ABED-MERAIM is an IEEE senior member and was associate editor for the IEEE
TRANSACTIONS ON SIGNAL PROCESSING (from 2001 to 2004).
\end{biographynophoto}
\begin{biographynophoto}{Sofi\`ene Affes} (S94, M95, SM04) received the Dipl\^ome d'Ing\'enieur in 
telecommunications in 1992, and the Ph.D. degree with honors in signal 
processing in 1995, both from the École Nationale Supérieure des 
Télécommunications (ENST), Paris, France.

He has been since with INRS-EMT, University of Quebec, Montreal, Canada, 
as a Research Associate from 1995 till 1997, as an Assistant Professor 
till 2000, then as an Associate Professor till 2009. Currently he is a 
Full Professor in the Wireless Communications Group. His research 
interests are in wireless communications, statistical signal and array 
processing, adaptive space-time processing and MIMO. From 1998 to 2002 
he has been leading the radio design and signal processing activities of 
the Bell/Nortel/NSERC Industrial Research Chair in Personal 
Communications at INRS-EMT, Montreal, Canada. Since 2004, he has been 
actively involved in major projects in wireless of PROMPT (Partnerships 
for Research on Microelectronics, Photonics and Telecommunications).

Professor Affes was the co-recipient of the 2002 Prize for Research 
Excellence of INRS. He currently holds a Canada Research Chair in 
Wireless Communications and a Discovery Accelerator Supplement Award 
from NSERC (Natural Sciences \& Engineering Research Council of Canada). 
In 2006, Professor Affes served as a General Co-Chair of the IEEE 
VTC¿2006-Fall conference, Montreal, Canada. In 2008, he received from 
the IEEE Vehicular Technology Society the IEEE VTC Chair Recognition 
Award for exemplary contributions to the success of IEEE VTC. He 
currently acts as a member of the Editorial Board of the IEEE 
Transactions on Wireless Communications, of the IEEE Transactions on 
Signal Processing, and of the Wiley Journal on Wireless Communications \& 
Mobile Computing.

\end{biographynophoto}

\end{document}